\documentclass{endm}
\usepackage{endmmacro}
\usepackage{graphicx}
\usepackage{amsmath, amssymb, amsfonts, mathrsfs, mathtools}

\def \Zl {{\mathbb Z}}
\def \Nl {{\mathbb N}}
\def \Rl {{\mathbb R}}
\def \Cl {{\mathbb C}}

\def \x {{\mathbf x}}
\def \y {{\mathbf y}}

\def \o {{\mathbf 0}}

\def \m {{\mathbf m}}
\def \n {{\mathbf n}}
\def \D {{\mathbf D}}
\def \d {{\mathbf d}}

\begin{document}

\begin{verbatim}\end{verbatim}\vspace{2.5cm}

\begin{frontmatter}

\title{A class of $\mathit{gcd}$-graphs having Perfect State Transfer}

\author{Hiranmoy Pal
\thanksref{myemail}}
\address{Department of Mathematics\\ IIT
Guwahati\\Guwahati, India}

\author{Bikash Bhattacharjya\thanksref{coemail}}
\address{Department of Mathematics\\ IIT
Guwahati\\Guwahati, India}
   \thanks[myemail]{Email:
   \href{hiranmoy@iitg.ernet.in} {\texttt{\normalshape
   hiranmoy@iitg.ernet.in}}} \thanks[coemail]{Email:
   \href{b.bikash@iitg.ernet.in} {\texttt{\normalshape
   b.bikash@iitg.ernet.in}}}

\begin{abstract}
Let $G$ be a graph with adjacency matrix $A$. The transition matrix corresponding to $G$ is defined by $H(t):=\exp{\left(itA\right)}$, $t\in\Rl$. The graph $G$ is said to have perfect state transfer (PST) from a vertex $u$ to another vertex $v$, if there exist $\tau\in\Rl$ such that the $uv$-th entry of $H(\tau)$ has unit modulus. The graph $G$ is said to be periodic at $\tau\in\Rl$ if there exist $\gamma\in\Cl$ with $|\gamma|=1$ such that $H(\tau)=\gamma I$, where $I$ is the identity matrix. A $\mathit{gcd}$-graph is a Cayley graph over a finite abelian group defined by greatest common divisors. In this paper, we construct classes of $\mathit{gcd}$-graphs having periodicity and perfect state transfer.

\end{abstract}


\end{frontmatter}

\section{Introduction}
Perfect state transfer in quantum communication networks was initially studied by S. Bose \cite{bose}. Cayley graphs appear frequently in communication networks. Cayley graphs are very well known class of vertex transitive graphs. A circulant graph is a cayley graph over a cyclic group. In \cite{saxe}, it was shown that a circulant graph is periodic if and only if the graph is integral. Also, in \cite{god}, Theorem $6.1$ implies that a vertex transitive graph admits PST if the graph is periodic. So it is more likely that a Cayley graph would exhibit PST only when the graph has integral spectrum. We therefore investigate PST on $\mathit{gcd}$-graphs as these Cayley graphs are known to have integral spectrum. Some work has already been done in this direction on integral circulant graphs and cubelike graphs. An integral circulant graph is a $\mathit{gcd}$-graph over a cyclic group. Characterization of circulant graphs having PST is given in \cite{mil}. A cubelike graph is a $\mathit{gcd}$-graph over the group $\Zl_{2}\times\ldots\times\Zl_{2}$. Perfect state transfer on cubelike graphs has been discussed in \cite{ber,chi}. We, however, find PST in more general class of $\mathit{gcd}$-graphs. We now define $\mathit{gcd}$-graph and restate some relevant results. 
\par Let $\left(\Gamma,+\right)$ be a finite abelian group and consider $S\subseteq\Gamma$ with the property that $-S=\left\lbrace -s:s\in S\right\rbrace=S$, \emph{i.e,} the set $S$ is symmetric. The Cayley graph over $\Gamma$ with the connection set $S$ is denoted by $Cay\left(\Gamma,S\right)$. The graph $Cay(\Gamma,S)$ has the vertex set $\Gamma$ where two vertices $a,b\in\Gamma$ are adjacent if and only if $a-b\in S$. If the additive identity $0\in S$ then $Cay\left(\Gamma,S\right)$ has a loop at each of its vertices. We use the convention that each loop \textbf{contributes one} to the corresponding diagonal entry of the adjacency matrix. The following lemma implies that the adjacency matrices of two Cayley graphs, defined over a fixed abelian group (finite), commute.
\begin{lem}\cite{ahm}\label{3ab}
Let $S$ and $T$ be symmetric subsets of a group $\Gamma$. If $g T = T g$ for every $g\in \Gamma$, then the adjacency matrices of the Cayley graphs $Cay(\Gamma, S)$ and $Cay(\Gamma, T )$ commute.
\end{lem}
Thus one can simply observe: if $\Gamma$ is an abelian group then adjacency matrices of any two cayley graphs over $\Gamma$ commute.\par
The greatest common divisor of two non-negative integers $m,n$ is denoted by $gcd(m,n)$. We use the convention that $gcd(0,n)=gcd(n,0)=n$ for every non-negative integer $n$. Consider two $r$-tuples of non-negative integers $\m=\left(m_1,\ldots, m_r\right)$ and $\n=\left(n_1,\ldots, n_r\right)$. For $i=1,\ldots,r$ suppose $d_i=gcd(m_i,n_i)$ and set $\d=\left(d_1,\ldots,d_r\right)$. We define $gcd(\m,\n)=\d$.\par
 The additive group of integers modulo $n$ is denoted by $\Zl_n$. Let $\left(\Gamma,+\right)$ be a finite abelian group and the cyclic group decomposition of $\Gamma$ is $$\Gamma=\Zl_{m_1}\oplus\ldots\oplus\Zl_{m_r},\text{ }m_i\geq 1 \text{ for } i=1,\ldots,r.$$
Suppose that $d_i$ is a positive divisor of $m_i$ for $i=1,\ldots, r$. For the divisor tuple $\d=\left(d_1,\ldots, d_r\right)$ of $\m=\left(m_1,\ldots, m_r\right)$, define \[S_\Gamma(\d)=\left\lbrace\x\in\Gamma : gcd(\x,\m)=\d\right\rbrace.\] Note that, for different tuples $\d$, the corresponding sets $S_\Gamma(\d)$ are disjoint. For a set of divisor tuples $\D$ of $\m$, define \[S_\Gamma(\D)=\bigcup\limits_{\d\in \D}S_\Gamma(\d)\]
The sets $S_\Gamma(\D)$ are called the gcd-sets of $\Gamma$. A Cayley graph $Cay(\Gamma,S)$ over a finite abelian group $\Gamma$ with a gcd-set $S$ is called a $\mathit{gcd}$-graph. If $\m\in\D$ then $S_\Gamma(\m)=\left\lbrace(0,\ldots,0)\right\rbrace$ and hence the corresponding $\mathit{gcd}$-graph has loops at each of its vertices. However, we mainly focus on finding PST on simple graphs. Therefore when considering PST on $\mathit{gcd}$-graphs we assume $\m\notin\D$. Since quantum particles cannot jump between disconnected components, PST is necessarily considered only on connected graphs. It is well known that a Cayley graph $Cay(\Gamma,S)$ is connected if and only if $S$ generates the group $\Gamma$. So whenever we consider PST on $\mathit{gcd}$-graphs we make sure that the set $\D$ generates the group $\Gamma$.\par
A circulant graph is a Cayley graph over the cyclic group $\Zl_n$. A graph with integer eigenvalues is called an integral graph. A circulant graph having integral spectrum is called an integral circulant graph. We denote the set of all divisors of a positive integer $n$ by $D_n$. The following result by \textit{W. So} characterizes the circulant graphs which are integral.
\begin{theorem}\label{1a}\cite{so}
A circulant graph $Cay(\Zl_n,S)$ is integral if and only if $S=\bigcup\limits_{d\in D}S_{\Zl_n}(d)$ for some set of divisors $D\subset D_{n}$.
\end{theorem}
Therefore an integral circulant graph can be defined by the order $n$ of the cyclic group and a set of divisors $D\subset D_n$ and thus it is denoted by $ICG_n(D)$. The following result characterizes all connected integral circulant graphs.
\begin{theorem}\label{1b}\cite{so}
An integral circulant graph $ICG_{n}(D)$, $D=\left\lbrace d_{1},\ldots, d_{n}\right\rbrace$, is connected if and only if $gcd(n,d_{1},\ldots, d_{n})=1$.
\end{theorem}
We now define Kronecker product of two given graphs $G$ and $H$. Let the graphs $G$ and $H$ have the vertex sets $U$ and $V$, respectively. The Kronecker product of $G$ and $H$, denoted by $G\times H$, has the vertex set $U\times V$. Two vertices $(u_{1},u_{2})$ and $(v_{1},v_{2})$ are adjacent in $G\times H$ whenever $u_{1}$ is adjacent to $v_{1}$ in $G$ and $u_{2}$ is adjacent to $v_{2}$ in $H$. If the graphs $G$ and $H$ have the adjacency matrices $A$ and $B$, respectively, then $G\times H$ has the adjacency matrix $A\otimes B$.
\par The next result enables us to find the transition matrix of Kronecker product of graphs when the transition matrix of one of the graphs is known.
\begin{prop}\cite{god,pal}\label{aa}
Let $G$ and $H$ be two graphs having adjacency matrices $A$ and $B$. Suppose the spectral decomposition of $B$ is $B=\sum\limits_{s=1}^{q}\mu_{s}F_{s}$. If $H(t)$ is the transition matrix of $G$ then $G\times H$ has the transition matrix $\sum\limits_{s=1}^{q} H(\mu_{s}t)\otimes F_{s}$.
\end{prop}
A few more information on perfect state transfer of Kronecker products can be found in \cite{god}.

\section{Perfect state transfer on cubelike graphs}
PST on simple cubelike graphs has already been discussed in \cite{ber,chi}. In this section we discuss some relevant results from \cite{chi}, which are also valid for looped cubelike graphs. We include this section for convenience as the results will be used to find PST in $\mathit{gcd}$-graphs.
\par A cubelike graph $X(C)$ is a Cayley graph over $\Zl_{2}^{n}$ with a connection set $C\subset\Zl_{2}^{n}$. For each $u\in\Zl_{2}^{n}$, the map $P_{u}:\Zl_{2}^{n}\longrightarrow\Zl_{2}^{n}$ defined by $P_{u}\left(x\right)=x+u$ is a permutation of the elements of $\Zl_{2}^{n}$ and hence it can be realized as a permutation matrix of appropriate order. It is easy to see that $P_{0}=I$ and  $P_{u}P_{v}=P_{u+v}$ which implies $P_{u}^{2}=I$. The following result finds the adjacency matrix of a cubelike graph.
\begin{lem}\label{2a}\cite{chi}
If $C\subseteq\Zl_{2}^{n}$ then the cubelike graph $X(C)$ has the adjacency matrix $A=\sum\limits_{u\in C}P_{u}$.
\end{lem}
Note that $\exp{\left(it(P_u+P_v)\right)}=\exp{(itP_u)}\exp{(itP_v)}$ as $P_{u}P_{v}=P_{u+v}=P_{v}P_{u}$ for all $u,v\in\Zl_n$. Using this property the transition matrix of a cubelike graph can be calculated as follows.
\begin{lem}\label{2b}\cite{chi}
If $H(t)$ is the transition matrix of $X(C)$ then $$H(t)=\prod\limits_{u\in C}\exp{(itP_{u})}.$$
\end{lem}
The next result determines periodicity and perfect state transfer in cubelike graphs at $t=\frac{\pi}{2}$. For other values of $t$, those properties are not completely characterized in cubelike graphs.
\begin{theorem}\label{2c}\cite{chi}
Let $C$ be a subset of $\Zl_{2}^{n}$ and let $\sigma$ be the sum of the elements of $C$. If $\sigma\neq \o$ then PST occurs in $X(C)$ from $u$ to $u+\sigma$ at time $\frac{\pi}{2}$. If $\sigma=\o$ then $X$ is periodic at $\frac{\pi}{2}$. 
\end{theorem}
We already have seen that $P_{u}^{2}=I$ and this implies $$\exp{(itP_u)}= I+itP_u-\frac{t^{2}}{2!}I-i\frac{t^{3}}{3!}P_u+\frac{t^{4}}{4!}I+\ldots = \cos{(t)}I+i\sin{(t)}P_u.$$ Also observe that if $\sigma$ is the sum of the elements in $C$ then $\prod\limits_{u\in C}P_{u}=P_{\sigma}$. By Lemma \ref{2b}, the transition matrix of $X(C)$ at $\frac{\pi}{2}$ can be evaluated as $$H\left(\frac{\pi}{2}\right)=\prod\limits_{u\in C}iP_{u}=i^{|C|}P_{\sigma}.$$ 

\begin{prop}\label{2d}
Let $X(C)$ be a cubelike graph with the connection set $C\subset\Zl_{2}^{n}$. Also assume that the sum of the elements of $C$ is $\o$ and $|C|\equiv 0\;(mod\;4)$. Then for every integral graph $G$, the transition matrix of $X(C)\times G$ at $\frac{\pi}{2}$ is the identity matrix.
\end{prop}
\begin{proof}
If $H(t)$ is the transition matrix of the cubelike graph $X(C)$ then we find that $$H\left(\frac{\pi}{2}\right)=\prod\limits_{u\in C}iP_{u}=i^{|C|}P_{\sigma},$$ where $\sigma$ is the sum of the elements in $C$. Now $\sigma=\o$ implies that $P_\sigma=I$ and therefore if $|C|\equiv0\;(mod\;4)$, then $H\left(\frac{\pi}{2}\right)=I$. Hence for every integer $\mu$, $H\left(\frac{\pi\mu}{2}\right)=\left(H\left(\frac{\pi}{2}\right)\right)^{\mu}=I$. Now consider $\sum\limits_{s=1}^{q}\mu_{s}F_{s}$ to be the spectral decomposition of $G$. Here the idempotents has the property that $\sum\limits_{s=1}^{q}F_{s}=I$. As the graph $G$ is assumed to be integral, the eigenvalues $\mu_s$ of $G$ are integers. By Proposition \ref{aa}, the transition matrix of $X(C)\times G$ at $\frac{\pi}{2}$ is obtained as
$$\sum\limits_{s=1}^{q} H\left(\frac{\pi\mu_s}{2}\right)\otimes F_{s}=I\otimes\sum\limits_{s=1}^{q} F_{s}=I\otimes I=I.$$ This proves our claim.
\end{proof}

It is clear from Proposition \ref{2d} that the graph $X(C)\times G$ is in fact periodic at $\frac{\pi}{2}$ whenever sum of the elements in $C$ is $\o$ along with $|C|\equiv0\;(mod\;4)$.

\section{Periodicity and Perfect State Transfer on gcd-graphs}

In this section, we find some periodic integral circulant graphs which in fact do not exhibit PST. Using those graphs, we construct some $\mathit{gcd}$-graphs allowing PST. The following result finds the transition matrix of the union of two edge disjoint Cayley graphs defined over a fixed finite abelian group.
\begin{prop}\label{3a}
Let $\Gamma$ be a finite abelian group and consider $S$ and $T$ to be disjoint symmetric subsets of $\Gamma$. If $H_{S}(t)$ and $H_{T}(t)$ are the transition matrices of $Cay(\Gamma, S)$ and $ Cay(\Gamma, T)$, respectively, then the transition matrix of $Cay(\Gamma, S\cup T)$ is $H_{S}(t)H_{T}(t)$.
\end{prop}
\begin{proof}
The adjacency matrix of the graph $Cay(\Gamma, S\cup T)$ is the sum of the adjacency matrices of the graphs $Cay(\Gamma, S)$ and $ Cay(\Gamma, T)$. As the group $\Gamma$ is abelian, by using Lemma \ref{3ab}, we find that the adjacency matrices of $Cay(\Gamma, S)$ and $ Cay(\Gamma, T)$ commute. This together with the fact that for any two square matrices $A, B$ with $AB=BA$, $\exp{(A+B)}=\exp{(A)}\exp{(B)}$, implies that $$H_{S\cup T}(t)=H_{S}(t)H_{T}(t).$$ This proves our claim.
\end{proof}

Now we show that certain integral circulant graph with an even number of vertices can be realized as a Kronecker product of a cubelike graph and an integral graph. We are going to use this representation of integral circulant graphs in finding periodic integral circulant graphs. We prove the result by using few techniques from \cite{wal}. 

\begin{prop}\label{1c}
Let $n\in2\Nl$ be such that $n=2^{\alpha}m$, where $\alpha\in\Nl$ and $m$ is an odd positive integer. Also consider $d=2^{\beta}m'$ to be a proper divisor of $n$. Let $X(C)$ be the cubelike graph with the connection set $$C=\left\lbrace \left(c_{0}, c_{1},\ldots,c_{\alpha-1}\right)\in\Zl_2^{\alpha}: c_{j}=0 \text{ for every } j<\beta \text{ and } c_{\beta}=1\right\rbrace.$$ Then we have the isomorphism $$Cay(\Zl_{n},S_{\Zl_{n}}(d))\cong X(C)\times  Cay(\Zl_{m},S_{\Zl_{m}}(m')).$$
\end{prop}
\begin{proof}
Assume that $G=X(C)\times Cay(\Zl_{m},S_{\Zl_{m}}(m'))$. As $gcd(2^{\alpha},m)=1$, by Chinese reminder theorem, every $z\in \Zl_{n}$ is uniquely determined by the congruences
$$z\equiv z'\; \left(mod\;2^{\alpha}\right),\; z\equiv z''\;\left(mod\; m\right).$$
Again, for every $z'\in\Zl_{2^{\alpha}}$, we have a unique 2-adic representation
\begin{eqnarray*}
z'=\sum\limits_{j=0}^{\alpha-1}z'_{j}2^{j}, \text{ where } z'_{j}\in\left\lbrace 0,1 \right\rbrace \text{ for } j=0,1,\ldots,\alpha-1.
\end{eqnarray*}
This leads to a bijection between the vertices of the graphs $Cay(\Zl_{n},S_{\Zl_{n}}(d))$ and $G$ which is given by $z\mapsto (\mathbf{z}',z'')$, where $\mathbf{z}'=(z'_{0},z'_{1},\ldots,z'_{\alpha-1})$. We claim that this is indeed an isomorphism. The vertices $x, y$ in $Cay(\Zl_{n},S_{\Zl_{n}}(d))$ are adjacent if and only if $gcd(x-y,n)=d$, which is equivalent to $gcd(x'-y',2^{\alpha})=2^{\beta}$ and $gcd(x''-y'',m)=m'$. Again $gcd(x'-y',2^{\alpha})=2^{\beta}$ if and only if $x'_{j}-y'_{j}=0$ for every $j<\beta$ and $x'_{\beta}-y'_{\beta}=1$ \emph{i.e.} $\mathbf{x}'-\mathbf{y}'\in C$. Hence $x$ and $y$ are adjacent in $Cay(\Zl_{n},S_{\Zl_{n}}(d))$ if and only if $(\mathbf{x}',x'')$ and $(\mathbf{y}',y'')$ are adjacent in $G$. This proves our claim. 
\end{proof}

In the next two lemma, we construct a class of integral circulant graphs which are periodic. These graphs are not necessarily connected. We then use Proposition \ref{3a} to construct periodic integral circulants which are also connected. 
\begin{lem}\label{3b}
If $n,d\in\Nl$ is such that $\frac{n}{d}\in8\Nl$ then the transition matrix of the integral circulant graph $ICG_{n}(\left\lbrace d\right\rbrace)$ at $\frac{\pi}{2}$ is the identity matrix. 
\end{lem}
 
\begin{proof}
By our assumption, we can write $n=2^{\alpha}m$, where $m$ is an odd positive integer and $\alpha\geq 3$. If $\frac{n}{d}\in8\Nl$ then $d=2^{\beta}m'$ for some $\beta\leq\alpha-3$. By using Proposition \ref{1c} we find that $ICG_{n}(\left\lbrace d\right\rbrace)$ is isomorphic to $X(C)\times Cay(\Zl_{m},S_{\Zl_{m}}(m'))$, where the connecting set $C$ in $X(C)$ is given by $$C=\left\lbrace \left(c_{0},\ldots, c_{\beta},\ldots,c_{\alpha-1}\right): c_{j}=0 \text{ for every } j<\beta\leq\alpha-3 \text{ and } c_{\beta}=1\right\rbrace.$$ It is clear that $|C|\equiv0\;(mod\;4)$ and $\sum\limits_{c\in C}c=\o$. Hence, by Proposition \ref{2d}, the transition matrix of $ICG_{n}(\left\lbrace d\right\rbrace)$ at $\frac{\pi}{2}$ is the identity matrix.
\end{proof}

\begin{lem}\label{3c}
If $n,d\in\Nl$ be such that $\frac{n}{d}\in8\Nl-4$ then the transition matrix of the integral circulant graph $ICG_{n}(\left\lbrace d,2d,4d\right\rbrace)$ at $\frac{\pi}{2}$ is the identity matrix.
\end{lem}
\begin{proof}
By our assumption, we can write $n=2^{\alpha}m$ where $m>1$ is an odd positive integer and $\alpha\geq 2$. If $\frac{n}{d}\in8\Nl-4$ then we have $d=2^{\alpha-2}k$, where $k$ divides $m$. Let $C_{1}=\left\lbrace (0,\ldots,0,1,0), (0,\ldots,0,1,1)\right\rbrace,\;C_{2}=\left\lbrace (0,0,\ldots,0,0,1)\right\rbrace,\;C_{3}=\left\lbrace (0,0,\ldots,0,0,0)\right\rbrace$. By Proposition \ref{1c}, we have the following isomorphisms.
\begin{enumerate}
\item $ICG_{n}(\left\lbrace 2^{\alpha-2}k\right\rbrace)\cong X(C_{1})\times Cay(\Zl_{m},S_{\Zl_{m}}(k)),$
\item $ICG_{n}(\left\lbrace 2^{\alpha-1}k\right\rbrace)\cong X(C_{2})\times Cay(\Zl_{m},S_{\Zl_{m}}(k)),$
\item $ICG_{n}(\left\lbrace 2^{\alpha}k\right\rbrace)\cong X(C_{3})\times Cay(\Zl_{m},S_{\Zl_{m}}(k)).$
\end{enumerate}
All three isomorphisms are caused by the same map that is described in the proof of Proposition \ref{1c}. Now consider $C=C_{1}\cup C_{2}\cup C_{3}$. We therefore have the following isomorphism $$ICG_{n}(\left\lbrace 2^{\alpha-2}k,2^{\alpha-1}k,2^{\alpha}k\right\rbrace)\cong X(C)\times Cay(\Zl_{m},S_{\Zl_{m}}(k)).$$ Note that the set $C$ contains exactly $4$ elements and also sum of the elements in $C$ is $\o$. Hence, by Proposition \ref{2d}, the transition matrix of the graph $ICG_{n}(\left\lbrace d,2d,4d\right\rbrace)$ at $\frac{\pi}{2}$ is the identity matrix.
\end{proof}
Recall that $D_n$ is the set of all positive divisors of $n$. For $n\in4\Nl$, suppose $\mathscr{D}_n$ is the collection of $D\subset D_n$ such that
\begin{enumerate}
\item $D=\tilde{D}\cup D'\cup 2D'\cup 4D'$ where $\tilde{D}\subseteq\left\lbrace d\in D_n: \frac{n}{d}\in8\Nl\right\rbrace$ and $D'\subseteq\left\lbrace d\in D_n: \frac{n}{d}\in 8\Nl-4\right\rbrace$,
\item $D$ generates the group $\Zl_n$.
\end{enumerate}
For example, we can consider $D=\left\lbrace 1,2,4\right\rbrace$ for $n\in8\Nl-4$. Also for $n\in8\Nl$, we can choose $D=\left\lbrace 1\right\rbrace$. It is now easy to see that, for large values of $n\in 4\Nl$, there are many such sets $D\in\mathscr{D}_n$. Notice that if $D\in\mathscr{D}_n$ then the graph $ICG_{n}(D)$ is connected as the set $D$ generates the group $\Zl_n$. The following result finds connected integral circulant graphs having periodicity.

\begin{theorem}\label{3d}
If $n\in4\Nl$ and $D\in\mathscr{D}_n$ then the transition matrix of the integral circulant graph $ICG_{n}(D)$ at $\frac{\pi}{2}$ is the identity matrix.
\end{theorem}
\begin{proof}
Here $D\in\mathscr{D}_n$ and therefore we can write $D=\tilde{D}\cup D'\cup 2D'\cup 4D'$, where $\tilde{D}\subseteq\left\lbrace d\in D_n: \frac{n}{d}\in8\Nl\right\rbrace$ and $D'\subseteq\left\lbrace d\in D_n: \frac{n}{d}\in 8\Nl-4\right\rbrace$. For $d\in\tilde{D}$ and $d'\in D'$, assume that $H_d(t)$ and $ H_{\left\lbrace d',2d',4d'\right\rbrace}(t)$ are the transition matrices of $ICG_{n}(\left\lbrace d\right\rbrace)$ and $ICG_{n}(\left\lbrace d',2d',4d'\right\rbrace)$, respectively. Applying Lemma \ref{3b} to $ICG_{n}(\left\lbrace d\right\rbrace)$, we have $H_d\left(\frac{\pi}{2}\right)=I$. Again, by using Lemma \ref{3c}, we find that $ H_{\left\lbrace d',2d',4d'\right\rbrace}\left(\frac{\pi}{2}\right)=I$. By Proposition \ref{3a}, the transition matrix of $ICG_{n}(D)$ at $\frac{\pi}{2}$ can be calculated as $$H_D\left(\frac{\pi}{2}\right)=\prod\limits_{d\in\tilde{D}}H_d\left(\frac{\pi}{2}\right)\prod\limits_{d'\in D'}H_{\left\lbrace d',2d',4d'\right\rbrace}\left(\frac{\pi}{2}\right)=I.$$ This proves the result.
\end{proof}
Note here that some of the graphs mentioned in Theorem \ref{3d} has loops. For example, if we choose $n=4$ then $\mathscr{D}_4$ contains the set $D=\left\lbrace 1,2,4\right\rbrace$. Clearly the graph $ICG_{4}(D)$ has loops. So to have a loopless periodic graph in Theorem \ref{3d}, we make sure that $n\notin D$.\par
Now we consider $\mathit{gcd}$-graphs over abelian groups. We intend to find some periodic $\mathit{gcd}$-graphs and then we add extra edges to those graphs to have PST.
\begin{theorem}\label{3f}
Let $\Gamma=\Zl_{m_1}\oplus\ldots\oplus\Zl_{m_r}$ be such that $m_i\equiv 0\;(mod\;4)$, for some $i=1,\ldots,r$. For $D\in\mathscr{D}_{m_i}$, consider the set $\D=\left\lbrace\left(1,\ldots,d_i,\ldots,1\right):d_i\in D\right\rbrace$. Then the graph $Cay(\Gamma,S_{\Gamma}(\D))$ is periodic at $\frac{\pi}{2}$.
\end{theorem}
\begin{proof}
It is enough to prove the result for $i=1$. Suppose that $$\Gamma'=\Zl_{m_2}\oplus\ldots\oplus\Zl_{m_r}.$$ We first show that $$Cay(\Gamma,S_{\Gamma}(\D))=ICG_{m_1}(D)\times Cay(\Gamma',S_{\Gamma'}(\left\lbrace \left(1,\ldots,1\right)\right\rbrace)).$$ Note that both these graphs have the same set of vertices. Also assume that $\x, \y\in\Gamma$, where $\x=\left(x_1,\ldots,x_r\right)$ and $\y=\left(y_1,\ldots,y_r\right)$. If $\d=\left(d_1,\ldots,d_r\right)\in\D$ then we have $gcd(\x-\y,\m)=\d$ if and only if $gcd(x_j - y_j,m_j)=d_j$ for each $j={1,\ldots,r}$. So $\x$ and $\y$ are adjacent in $Cay(\Gamma,S_{\Gamma}(\D))$ if and only if $x_1$ is adjacent to $y_1$ in $ICG_{m_1}(D)$ and $\left(x_2,\ldots,x_r\right)$ is adjacent to $\left(y_2,\ldots,y_r\right)$ in $Cay(\Gamma',S_{\Gamma'}(\left\lbrace \left(1,\ldots,1\right)\right\rbrace))$. Consequently, the vertices $\x$ and $\y$ are adjacent in $ICG_{m_1}(D)\times Cay(\Gamma',S_{\Gamma'}(\left\lbrace\left(1,\ldots,1\right)\right\rbrace))$.\par Suppose the spectral decomposition of adjacency matrix of the graph $Cay(\Gamma',S_{\Gamma'}(\left\lbrace \left(1,\ldots,1\right)\right\rbrace))$ is $\sum\limits_{s=1}^{q}\mu_{s}F_{s}$. Since $Cay(\Gamma',S_{\Gamma'}(\left\lbrace \left(1,\ldots,1\right)\right\rbrace))$ is a $\mathit{gcd}$-graph, it has integral spectrum. Therefore the eigenvalus $\mu_s$ are integers for $s=1,\ldots,q$. If $H(t)$ is the transition matrix of $ICG_{m_1}(D)$ then by using Theorem \ref{3d}, we have $H(\frac{\pi}{2})=I$. Again applying Theorem \ref{aa}, the transition matrix of $ICG_{m_1}(D)\times Cay(\Gamma',S_{\Gamma'}(\left\lbrace\left(1,\ldots,1\right)\right\rbrace))$ at $\frac{\pi}{2}$ can be evaluated as
$$\sum\limits_{s=1}^{q} H\left(\frac{\pi\mu_{s}}{2}\right)\otimes F_{s}= \sum\limits_{s=1}^{q} H\left(\frac{\pi}{2}\right)^{\mu_{s}}\otimes F_{s}=I\otimes\sum\limits_{s=1}^{q}F_s=I.$$ This implies that the graph $Cay(\Gamma,S_{\Gamma}(\D))$ is periodic at $\frac{\pi}{2}$.
\end{proof}
We have constructed a class of $\mathit{gcd}$-graphs which are periodic. Now we add some edges to these graphs to create a class of $\mathit{gcd}$-graphs allowing PST. 
\begin{theorem}\label{3e}
If the conditions of Theorem \ref{3f} are satisfied then perfect state transfer occurs in the cayley graphs with gcd-sets $\D\cup\left\lbrace(m_1,\ldots,\frac{m_i}{2},\ldots,m_r)\right\rbrace$ as well as $\D\cup\left\lbrace(m_1,\ldots,\frac{m_i}{4},\ldots,m_r)\right\rbrace$.
\end{theorem}
\begin{proof}
Observe that the subgroup of $\Gamma$ generated by $(m_1,\ldots,\frac{m_i}{2},\ldots,m_r)$ is isomorphic with $\Zl_2$. The subgroup has $\frac{m_1\cdots m_r}{2}$ right cosets and therefore the Cayley graph over the gcd-set $\left\lbrace (m_1,\ldots,\frac{m_i}{2},\ldots,m_r)\right\rbrace$ has $\frac{m_1\cdots m_r}{2}$ components and each component is isomorphic with the complete graph $K_2$. The graph $K_{2}$ is known to have PST at $\frac{\pi}{2}$ and hence disjoint union of any number of copies of $K_{2}$ also exhibit PST. Hence, by Proposition \ref{3a}, the cayley graph over $\Gamma$ with the gcd-set $\D\cup\left\lbrace(m_1,\ldots,\frac{m_i}{2},\ldots,m_r)\right\rbrace$ admits PST.\par Similarly, the Cayley graph over the gcd-set $\left\lbrace (m_1,\ldots,\frac{m_i}{4},\ldots,m_r)\right\rbrace$ can be realized as the disjoint union of $\frac{m_1\cdots m_r}{4}$ copies of the cycle $C_{4}$. The graph $C_{4}$ is also known to have PST at $\frac{\pi}{2}$ and therefore PST occurs on disjoint union of any number of copies of $C_{4}$. Hence, by Proposition \ref{3a}, the cayley graph over $\Gamma$ with the gcd-set $\D\cup\left\lbrace(m_1,\ldots,\frac{m_i}{4},\ldots,m_r)\right\rbrace$ also exhibits PST.
\end{proof}
Thus we have constructed a class of $\mathit{gcd}$-graphs allowing PST apart from the circulants or cubelike graphs. The following result, which is a direct consequence of Theorem \ref{3e}, produces some integral circulant graphs exhibiting PST. In \cite{mil}, it was in fact shown that these are the only integral circulant graphs allowing PST.
\begin{cor}
Let $n\in 4\Nl$ and the set $D\in\mathscr{D}_n$ be such that $n\notin D$. Also consider $D'=D\cup\left\lbrace \frac{n}{2}\right\rbrace$ and $D''=D\cup\left\lbrace \frac{n}{4}\right\rbrace$. Then the graphs $ICG_{n}(D')$ and $ICG_{n}(D'')$ exhibit PST at $\frac{\pi}{2}$.
\end{cor}
We now give a simple characterization of $\mathit{gcd}$-graphs allowing perfect state transfer. Let the group $\Gamma=\Zl_{m_1}\oplus\ldots\oplus\Zl_{m_r}$, $r>1$ and $m_i\equiv 0\;(mod\;4)$ for some $i$. If $m_i\in8\Nl$ then choose $D=\left\lbrace1\right\rbrace$, otherwise consider $D=\left\lbrace1,2,4\right\rbrace$. Note that $D\in\mathscr{D}_{m_i}$ and therefore applying Theorem \ref{3f} and Theorem \ref{3e}, we can conclude the following:
\begin{theorem}\label{3g}
If $\Gamma=\Zl_{m_1}\oplus\ldots\oplus\Zl_{m_r}$, $r>1$ and $m_i\equiv 0\;(mod\;4)$ for some $i=1,\ldots,r$ then there exist a set of divisors $\D$ such that $Cay(\Gamma,S_{\Gamma}(\D))$ exhibits Perfect state transfer.
\end{theorem}
Notice that the complete graph $K_2$ is a $\mathit{gcd}$-graph over $\Zl_2$ and it is well known that $K_2$ exhibits PST. So the conditions in Theorem \ref{3g} are not necessary.

\section{Conclusions}
Perfect state transfer is highly desirable in quantum communication networks. Communication networks are frequently modeled on Cayley graphs. Thus it is useful to find PST in $\mathit{gcd}$-graphs. We already found a class of integral circulant graphs which are periodic and using these graphs we have constructed a class of $\mathit{gcd}$-graphs allowing PST. Also we have given a sufficient condition for PST in $\mathit{gcd}$-graphs. Further research can be done in this direction. It will be interesting to find all the $\mathit{gcd}$-graphs over a given abelian group $\Gamma$ having perfect state transfer.

\end{document}